\documentclass{article}

\usepackage{amsmath, amsthm, amssymb, amstext, amsfonts}
\usepackage{enumerate}
\usepackage{graphicx}
\usepackage[applemac]{inputenc}

\theoremstyle{plain}

\newtheorem{thm}{Theorem}[section]
\newtheorem{lem}[thm]{Lemma}

\newtheorem{prop}[thm]{Proposition}

\newtheorem{clm}[thm]{Claim}

\theoremstyle{definition}

\newcommand{\Lra}{\ensuremath{\Leftrightarrow}}

\newcommand{\sm}{\ensuremath{\setminus}}

\newcommand{\isom}{\ensuremath{\cong}}

\newcommand{\Aut}{\textnormal{Aut}}

\newcommand{\es}{\ensuremath{\emptyset}}

\newcommand{\sub}{\subseteq}

\newcommand{\nat}{{\mathbb N}}

\newcommand{\AF}{\ensuremath{\mathcal A}}

\begin{document}

\title{The classification of finite and locally finite connected-homogeneous digraphs}
\author{Matthias Hamann\bigskip \\Fachbereich Mathematik\\Universit\"at Hamburg}
\date{\today}
\maketitle

\begin{abstract}
We classify the finite connected-homogeneous digraphs, as well as the infinite such digraphs with precisely one end.
This completes the classification of all the locally finite connected-homo\-ge\-neous digraphs.
\end{abstract}

\section{Introduction}

A graph is called {\em homogeneous} if every isomorphism between two finite induced subgraphs extends to an automorphism of the entire graph.
The countable homogeneous graphs were classified in~\cite{Gard-HomogeneousGraphs,LW-CountUltrahomGraphs}.
Weakening the assumptions of homogeneity so that only isomorphisms between finite {\em connected} induced subgraphs have to extend to automorphisms leads to the notion of {\em connected-homogeneous} graphs, or simply {\em C-homogeneous} graphs.
These graphs were classified in~\cite{Enomoto,Gard-HomConditions,GrayMacpherson,HP-Transitivity,HedmanPong}.

For directed graph, or digraphs, the same notions of homogeneity and C-homoge\-neity apply.
The homogeneous digraphs were classified in \cite{Cherlin-CountHomDigraphs,L-FiniteHomDigraphs,L-Tournaments}.
Of the C-homo\-geneous digraphs only those that have more than one end have been classified \cite{GM-CHomDigraphs,HH-ConHomDigraphs}.
This paper completes the classification of locally finite C-homoge\-neous digraphs, by describing those that are finite or have precisely one end (Theorem~\ref{thm_Main1}).

Undirected locally finite C-homogeneous graphs, as is well-known, cannot have precisely one end (see~\cite{DistanceTransitive}).
Directed such graphs can; but they have a very restricted structure.
We shall see in Section~\ref{sec_Imprimitive} that these digraphs are quotients of one particular locally finite C-homogeneous digraph with infinitely many ends, the digraph $T(2)$.
This is the digraph in which every vertex is a cut vertex and lies on precisely two directed triangles.
Some of the finite examples are also quotients of $T(2)$.
It turns out that all the other finite connected C-homogeneous digraphs have their origin in the finite homogeneous digraphs; they are canonical generalizations of the homogeneous digraphs.
See Section~\ref{sec_NonIndCase} and Section~\ref{sec_IndCase} for more details.

Recall that every connected locally finite transitive (di)graph has either none, one, two, or infinitely many ends, see~\cite{DJM}.
Together with the classification by Gray and M\"oller \cite{GM-CHomDigraphs} of the two-ended digraphs and the classification by Hamann and Hundertmark~\cite{HH-ConHomDigraphs} of the infinitely-ended digraphs, our results thus complete the classification of all the locally finite C-homogeneous digraphs.

\section{Preliminaries}

\subsection{Definitions}

A {\em digraph} $D=(VD,ED)$ consists of a non-empty set $VD$ of {\em vertices} and an asymmetric, i.e.\ irreflexive and anti-symmetric, relation $ED$ on $VD$, its {\em edges}.
For $(x,y)\in ED$ we simply write $xy\in ED$ and say that the edge $xy$ is directed from $x$ to~$y$.
Then the vertices $x$ and $y$ are {\em adjacent}.

For $x\in VD$ we denote with $N^+(x)$ the {\em out-neighborhood} $\{y\mid xy\in ED\}$, with $N^-(x)$ the {\em in-neighborhood} $\{y\mid yx\in ED\}$, and with $N(x)$ the {\em neighborhood} $N^+(x)\cup N^-(x)$ of~$x$.
The {\em out-degree} $d^+(x)$ of~$x$ is the cardinality of $N^+(x)$, the {\em in-degree} $d^-(x)$ is the cardinality of $N^-(x)$, and the {\em degree} $d(x)$ is the cardinality of~$N(x)$.
If $D$ is a transitive digraph, then we denote with $d^+,d^-$ the value of $d^+(x),d^-(x)$, respectively, for all $x\in VD$.
Every element of $N^+(x)$ is called a {\em successor} of~$x$ and every element of $N^-(x)$ is called a {\em predecessor} of~$x$.

A {\em ($k$)-arc} is a directed path (of length $k$).
An {\em ancestor} ({\em descendant}) of a vertex $x$ is any vertex $y$ for which there exists an arc from $y$ to~$x$ (from $x$ to~$y$).
A {\em walk} is a sequence $x_0x_1\ldots x_n$ of vertices such that $x_i$ and $x_{i+1}$ are adjacent for all $0\leq i<n$and it is an {\em alternating} walk if we have $x_{i-1}\in N^+(x_i)\Lra x_{i+1}\in N^+(x_i)$ for all $1\leq i\leq n-1$.
If two edges lie on a common alternating walk then they {\em reachable} from each other.
This defines an equivalence relation, the {\em reachability relation}, which we denote with $\AF$.
For the equivalence class of an edge $e$ we write $\AF(e)$ and call the by $\AF(e)$ induced subdigraph $\langle\AF(e)\rangle$ the {\em reachability digraph} of~$D$ that contains $e$.
If $D$ is {\em $1$-arc transitive}, that is $\Aut(D)$ is transitive on the $1$-arcs of~$D$, then all reachability digraphs of~$D$ are isomorphic and we denote with $\Delta(D)$ the corresponding digraph.

The reachability digraph of an edge $e$ is a {\em bipartite reachability digraph} if it is bipartite, if one class of this bipartition has empty in-neighborhood in $\langle\AF(e)\rangle$ and if the other class has empty out-neighborhood.

The following proposition is due to Cameron et.al.\ \cite[Proposition 1.1]{CPW}.

\begin{prop}\label{prop_CPW}
Let $D$ be a connected $1$-arc transitive digraph.
Then $\Delta(D)$ is $1$-arc transitive and connected.
Further, either
\begin{enumerate}[{\em (a)}]
\item $\AF$ is the universal relation on $ED$ and $\Delta(D)=D$, or
\item $\Delta(D)$ is a bipartite reachability digraph.\qed
\end{enumerate}
\end{prop}

We need some notations for infinite (di)graphs.
Let $G$ be a graph.
A {\em ray} in~$G$ is a one-way infinite path.
Two rays are {\em equivalent} if for every finite set $S$ of vertices both rays lie eventually in the same component of $G-S$.
This property is an equivalence relation whose equivalence classes are called the {\em ends} of~$G$.
The {\em ends} of a digraph are the ends of the underlying undirected graph.

\bigskip

In the following we describe some classes of digraphs that occur during the investigation of locally finite C-homogeneous digraphs.
With $C_m$ we usually denote directed cycles of length~$m$.
But if it is obvious from the context that we are considering a subdigraph of a bipartite reachability digraph, then we also use $C_m$ to denote a cycle in that reachability digraph.
Cycles of length $3$ are {\em triangles}.

A vertex set is {\em independent} if no two of its vertices are adjacent.
The digraph $\bar{K}_n$ is the empty digraph on~$n$ vertices.

For two digraphs $D,D'$ we denote with $D[D']$ the {\em lexicographic product} of $D$ and $D'$, that is the digraph with vertex set $VD\times VD'$ and edge set$$\{(x,y)(x',y')\mid xx'\in ED \text{ or } (x=x' \text{ and } yy'\in ED')\}.$$

The {\em complete bipartite} digraph is that bipartite digraph that contains all edges from $A$ to~$B$ for the bipartition $A\cup B$.
The {\em (directed) complement of a perfect matching} $CP_k$ is the digraph obtained from the complete bipartite digraph where a perfect matching between $A$ and $B$ is removed.

Let $Y_k$ be the digraph with vertex set $V_1\cup V_2\cup V_3$ where the $V_i$ denote pairwise disjoint sets of the same cardinality~$k$.
There are no edges $xy$ with $xy\in V_i$ for $i=1,2,3$ and the subdigraphs $D[V_i,V_{i+1}]$ (for $i=1,2,3$ with $V_4=V_1$) are isomorphic to complements of perfect matchings such that all edges are directed from $V_i$ to $V_{i+1}$ and such that the tripartite complement of~$D$ is the disjoint union of copies of~$C_3$, where the {\em tripartite complement} of~$D$ is the digraph
$$(VD,(\bigcup_{i=1,2,3}(V_i\times V_{i+1}))\setminus ED).$$

Let $\sim$ be an equivalence relation on a digraph $D$.
With $D_\sim$ we denote the digraph whose vertex set is the set of equivalence classes and with edges $XY$ whenever there are representatives $x\in X,y\in Y$ such that $xy\in ED$.
This is not a digraph in our restrictive meaning because it may have loops or for an edge $xy$ there might also exist the edge $yx$.
However, we just consider such equivalence relations that makes $D_\sim$ to a digraph, that means its adjacency relation is irreflexive and anti-symmetric.

\subsection{Group action}

Let $\Gamma$ be a group acting on a digraph $D$ and let $U\subseteq VD$.
We denote with $\Gamma_{U}$ the {\em (pointwise) stabilizer} of $U$, that is the subgroup of~$\Gamma$ that fixes each element of~$U$.
The same notion holds for an edge $e\in ED$ or a single vertex $x\in VD$.
If $\Gamma$ fixes the set $U$ setwise, then we denote with $\Gamma^U$ all the automorphisms of~$U$ that are obtained by restricting elements of~$\Gamma$ on~$U$.

For the following well-know proposition see for example \cite{Wielandt} or \cite[3.1.2]{Stellmacher}.

\begin{prop}\label{prop_Wielandt}
Every subgroup of $S_n$ with $n\in\nat$ is equal to $A_n$ or has index at least $n$.\qed
\end{prop}

\subsection{Homogeneous digraphs}

In this section we briefly recall the classification result of Lachlan for homogeneous digraphs \cite{L-FiniteHomDigraphs}.
Let $H$ be the digraph depicted in Figure~\ref{pic_H}.

\begin{figure}[h]
\begin{center}
\includegraphics[width=.50\textwidth]{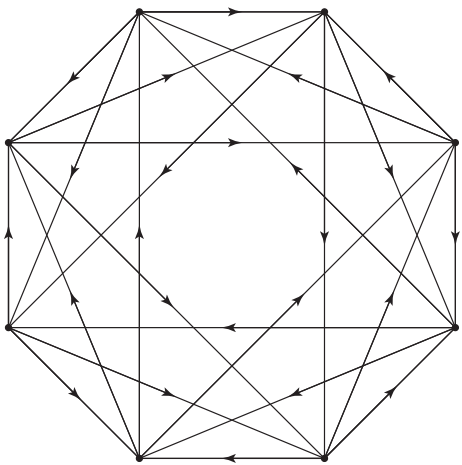}
\caption[Figure 1]{The digraph $H$}\label{pic_H}
\end{center}\end{figure}

\begin{thm}\label{thm_HomDigraphs}
{\em (Lachlan~\cite[Theorem~1]{L-FiniteHomDigraphs})}
A finite digraph is homogeneous if and only if it is isomorphic to one of the following digraphs:
\begin{enumerate}[{\em (i)}]
\item the $C_4$;
\item a $\bar{K}_n$ for an $n\geq 1$;
\item a $\bar{K}_n[C_3]$ for an $n\geq 1$;
\item a $C_3[\bar{K}_n]$ for an $n\geq 1$;
\item the digraph $H$.\qed
\end{enumerate}
\end{thm}

\section{The non-independent case}\label{sec_NonIndCase}

It is a straightforward argument that the out-neighborhood as well as the in-neighborhood of any vertex of a C-homogeneous digraph has to be a homogeneous digraph.
We investigate which of the homogeneous digraphs of Theorem~\ref{thm_HomDigraphs} may occur as a subdigraph induced by $N^+(x)$ or by $N^-(x)$ for a vertex $x\in VD$.
In this section we take a look at those cases that contain an edge and show that there is precisely one such case that may occur.
This case is a generalization of the digraph $H$ that occurs in the case (v) of Theorem~\ref{thm_HomDigraphs}.

Let $x$ be a vertex of a connected locally finite C-homogeneous digraph.
Our first aim is to show that $N^+(x)$ and $N^-(x)$ are both not isomorphic to~$H$.
Therefore, we define a {\em dominated} directed triangle to be a digraph that is isomorphic to a directed triangle together with a vertex that sends edges to all its vertices (Figure~\ref{pic_D3}).

\begin{figure}[h]
\begin{center}
\includegraphics[width=.250\textwidth]{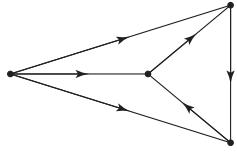}
\caption[Figure 1]{A dominated directed triangle}\label{pic_D3}
\end{center}\end{figure}

\begin{lem}\label{lem_NoH}
For every connected locally finite C-homogeneous digraph $D$ there is $N^+(x)\not\isom H$ and $N^-(x)\not\isom H$ for all $x\in VD$.
\end{lem}

\begin{proof}
Let $x\in VD$ and suppose by symmetry that $N^+(x)\isom H$.
Then there is a dominated directed triangle embedded in $N^-(y)$ for all $y\in N^+(x)$.
Hence we also have $N^-(x)\isom H$.

If two vertices $x,y$ are adjacent, say $xy\in ED$, then $|N^+(x)\cap N^+(y)|\leq 3$ since $N^+(y)\isom H$.
Furthermore, there exists a vertex $z\in N^-(y)\cap N^+(x)$.

\begin{clm}\label{clm_NoHClm1}
No neighbor of~$y$ lies in $N^+(x)\cap N^+(z)$.
\end{clm}

\begin{proof}[Proof of Claim~\ref{clm_NoHClm1}]
Suppose that there is a vertex $a\in N^+(x)\cap N^+(z)\cap N(y)$.
By mapping $D[x,y,z]$ onto $D[x,a,z]$ by an automorphism of~$D$, we get recursively a directed cycle in $N^+(x)\cap N^+(z)$.
We already mentioned that $|N^+(x)\cap N^+(z)|\leq 3$. Hence there is a directed triangle in $N^+(x)\cap N^+(z)$.
Let $v_1$ be a vertex in $N^+(x)$ that has two neighbors in $N^+(x)\cap N^+(z)$, let $v_2$ be a vertex in $N^+(x)$ with $N^+(x)\cap N^+(z)\sub N^+(v_2)$, and let $v_3$ be a vertex in $N^+(x)$ with $N^+(x)\cap N^+(z)\sub N^-(v_3)$.
Such vertices exist because $N^+(x)\isom H$.
Then either two of these vertices are adjacent to~$z$---and hence lie in $N^-(z)$---or two of them are not adjacent to $z$.
Let $v_i,v_j$ ($i\neq j$) be two vertices either both of the first or both of the second kind.
Then $D[z,x,v_i]\isom D[z,x,v_j]$, and thus there is an automorphism of~$D$ mapping the first onto the second subdigraph.
But this is a contradiction by the choice of $v_i$ and $v_j$.
\end{proof}

As $N^+(x)\isom H$, there are two out-neighbors of~$z$ that are adjacent to~$y$ in contradiction to Claim~\ref{clm_NoHClm1}.
Thus the lemma is proved.
\end{proof}

The next case that we exclude is that neither the out- nor the in-neighborhood induces a subdigraph isomorphic to~$C_4$.

\begin{lem}\label{lem_NoC4}
Let $D$ be a connected locally finite C-homogeneous digraph and let $x\in VD$.
Then $N^+(x)\not\isom C_4$ and $N^-(x)\not\isom C_4$.
\end{lem}

\begin{proof}
By regarding the digraph whose edges are directed in the inverse way, if necessary, we may suppose that $N^+(x)\isom C_4$.
Let us denote with $v_1,\ldots,v_4$ the four vertices in $N^+(x)$ with $v_iv_{i+1}\in ED$ for $1\leq i\leq 3$ and $v_4v_1\in ED$.
Since $x,v_4\in N^-(v_1)$ and since $N^-(v_1)$ is homogeneous, there is another vertex in $N^-(v_1)$ distinct from both $x$ and $v_4$.

We know by Lemma~\ref{lem_NoH} that $N^-(v_1)\not\isom H$.

\begin{clm}\label{clm_NoC4Clm1}
There is no other vertex than $x$ in $N^-(v_1)\cap N^-(v_2)$.
\end{clm}

\begin{proof}[Proof of Claim~\ref{clm_NoC4Clm1}]
Let us suppose that there is a vertex $y\in N^-(v_1)\cap N^-(v_2)$.
Then an immediate consequence of the C-homogeneity is $N^+(x)=N^+(y)$.
But then neither $xy$ nor $yx$ can be an edge of~$D$.
The subdigraph induced by $\{x,y,4\}$ is a subdigraph of $N^-(v_1)$ and thus $N^-(v_1)\isom \bar{K}_n[C_3]$ with $n>1$.
Then there is $z\in N^+(x)\cap N^-(v_1)$ which is distinct from~$v_4$.
This is not possible and hence no such $y$ exists.
\end{proof}

\begin{clm}\label{clm_NoC4Clm2}
There is no vertex in $N^-(v_1)\cap N^+(v_2)$.
\end{clm}

\begin{proof}[Proof of Claim~\ref{clm_NoC4Clm2}]
Suppose that there is a vertex $y\in N^-(v_1)\cap N^+(v_2)$.
If $y$ is neither adjacent to~$x$ nor to~$v_4$, then $N^-(v_1)$ has to be isomorphic to $\bar{K}_n[C_3]$ with $n>1$.
Then there is an automorphism $\alpha$ of~$D$ with $v_4^\alpha= v_4$, $v_1^\alpha=v_1$, and $v_2^\alpha=y$ and hence $x\neq x^\alpha\in N^-(v_1)\cap N^-(v_4)$.
This contradicts Claim~\ref{clm_NoC4Clm1} with $v_4$ and $v_1$ instead of $v_1$ and $v_2$.
So $y$ is adjacent to at least one of $x$ and $v_4$.

If $y$ is adjacent to $x$ but not to $v_4$, then $N^-(v_1)\isom C_4$ or $N^-(v_1)\isom H$ since an induced path of length $2$ embeds into $N^-(v_1)$.
As we already saw, only the first case can occur and then there is an automorphism $\alpha$ of~$D$ with $v_4^\alpha=y$, $v_1^\alpha=v_1$, and $v_2^\alpha=v_2$.
So $x\neq x^\alpha$ and $x^\alpha$ is a second vertex in $N^-(v_1)\cap N^-(v_2)$, which is impossible by Claim~\ref{clm_NoC4Clm1}.

If $y$ is adjacent to~$v_4$ but not to~$x$, then we distinguish two cases: in the first one $yv_4\in ED$. But then by C-homogeneity applied to $D[y,x,v_4]$ and $D[y,x,v_1]$ also $v_2\in N^+(y)$ contrary to the case we are discussing.
In the second case we have $v_4y\in ED$ and thus $N^-(v_1)\isom C_4$.
Then there has to be a vertex $z\in N^-(v_1)\sm\{v_4,x,y\}$.
If $z$ is not adjacent to~$v_2$, then there is an automorphism of~$D$ that maps $D[v_2,v_1,z]$ onto $D[v_2,v_1,v_4]$. Since this automorphism cannot fix $x$, the image of~$x$ also lies in $N^-(v_1)\cap N^-(v_2)$ contrary to Claim~\ref{clm_NoC4Clm1}.
If $v_2z\in ED$, then there is an automorphism of~$D$ that maps the cycle $D[v_2,y,v_1]$ onto $D[v_2,z,v_1]$. This is again a contradiction and the final contradiction in the case that $y$ is adjacent to~$v_4$ but not to~$x$ is given directly by Claim~\ref{clm_NoC4Clm1} since, if $zv_2\in ED$, then $z\in N^-(v_1)\cap N^-(v_2)$.

Let us now consider the case that both $x$ and $v_4$ are adjacent to~$y$.
By the same arguments as above there has to be $v_4 y\in ED$ and not $yv_4\in ED$.
By C-homogeneity we have $yv_3\in ED$ and since $y\notin N^+(x)$, we have $yx\in ED$.
But then $D[v_1,x,v_3]$ is a subdigraph of $N^+(y)$ but this digraph cannot be embedded into a $C_4$ and thus we just have proved the final contradiction of this claim.
\end{proof}

\begin{clm}\label{clm_NoC4Clm3}
There is no vertex in $N^-(v_1)\cap N^+(v_4)$ that is not adjacent to~$v_2$.
\end{clm}

\begin{proof}[Proof of Claim~\ref{clm_NoC4Clm3}]
Let us suppose that there exists $y\in N^-(v_1)\cap N^+(v_4)$ such that $y$ is not adjacent to~$v_2$.
Then $v_3$ is not adjacent to~$y$, too, and hence there is an automorphism $\alpha$ of~$D$ that maps $D[v_3,v_4,v_1]$ onto $D[v_3,v_4,y]$.
Since $y\notin N^+(x)$, we have $x\neq x^\alpha\in N^-(v_3)\cap N^-(v_4)$ and thus a contradiction to Claim~\ref{clm_NoC4Clm1}.
\end{proof}

\begin{clm}\label{clm_NoC4Clm4}
There is $N^-(v_1)\cap N^+(v_2)\neq\es$ or $N^-(v_1)\cap N^+(v_4)\neq\es$.
\end{clm}

\begin{proof}[Proof of Claim~\ref{clm_NoC4Clm4}]
Suppose that both intersections are empty.
Let $y\in N^-(v_1)$ with $x\neq y\neq v_4$.
If $x$ and $y$ are not adjacent, then $N^-(v_1)$ has to be isomorphic to $\bar{K}_n[C_3]$ for an $n>1$.
Hence there is $z\in N^+(v_4)\cap N^-(x)\cap N^-(v_1)$.
This is a direct contradiction to the assumptions.
Thus $x$ and $y$ has to be adjacent and hence $yx\in ED$.
So there is an induced path of length $2$ in $N^-(v_1)$ and thus $N^-(v_1)\isom C_4$ or $N^-(v_1)\isom H$, whereas the second case cannot occur by Lemma~\ref{lem_NoH}.
So $N^-(v_1)\isom C_4$.
Then both $D[v_4,v_1,v_2]$ and $D[y,v_1,v_2]$ are isomorphic subdigraphs of~$D$ and thus there is an automorphism $\alpha$ of~$D$ that fixes $v_1$ and $v_2$ and maps $v_4$ onto $y$.
We conclude $x^\alpha\in N^-(v_1)\cap N^-(v_2)$ which is untenable because of Claim~\ref{clm_NoC4Clm1}.
\end{proof}

By all the claims we showed that there is no vertex in $N^-(v_1)$ distinct from $x$ and from $v_4$ in contradiction to the homogeneity of $N^-(v_1)$ by Theorem~\ref{thm_HomDigraphs}.
Thus we proved Lemma~\ref{lem_NoC4}.
\end{proof}

\begin{lem}\label{lem_KnC3ImpliesH}
Let $D$ be a connected C-homogeneous digraph with $N^+(x)\isom \bar{K}_n[C_3]$ and $N^-(x)\isom \bar{K}_m[C_3]$ for all $x\in VD$ and with $m,n\geq 1$.
Then $m=n=1$.
\end{lem}

\begin{proof}
Let $xy\in ED$.
Then there exists $z\in N^-(y)\cap N^-(x)$.
By regarding $N^-(y)$, we obtain an $a\in N^-(y)\cap N^+(x)$ with $az\in ED$.
Let $b$ be the third vertex of $N^+(x)$ in that isomorphic image of~$C_3$, that contains $y$ and $a$.
We have neither $zb$ nor $bz$ in $ED$ since otherwise there is an edge either in $N^+(x)\cap N^+(z)$ or in $N^+(x)\cap N^-(z)$ and by applying the C-homogeneity we obtain the whole isomorphic image of $C_3$, $D[a,b,y]$, in $N^+(x)\cap N^+(z)$ or in $N^+(x)\cap N^-(z)$ which is impossible.

Let us suppose that $n>1$.
Then there exists a vertex $y'\in N^+(x)$ that is distinct from $a,b$, and $y$.
Then there is a vertex $v\in\{a,b,y\}$ such that $D[z,x,v]\isom D[z,x,y']$ and hence the isomorphic image of~$C_3$ in $N^+(x)$ that contains $y'$ contains a vertex of $N^+(z)$.
We may suppose that $y'\in N^+(z)$.
But then $D[y,x,y']$ is a digraph that cannot be embedded into $N^+(z)$.
So $n\not > 1$.
By a symmetric argument we also have $m=1$.
\end{proof}

\begin{lem}\label{lem_C3KnImpliesHKn}
Let $D$ be a connected locally finite C-homogeneous digraph and $x\in VD$.
If $N^+(x)\isom C_3[\bar{K}_n]$ or if $N^-(x)\isom C_3[\bar{K}_n]$ for an $n\geq 1$, then there is $D\isom H[\bar{K}_n]$.
\end{lem}

\begin{proof}
We assume that $N^+(x)\isom C_3[\bar{K}_n]$ for an $n\geq 1$.
Let $y\in N^+(x)$.
Then $x$ together with $n$ vertices of $N^+(x)$ lie in $N^-(y)$ and hence $N^-(y)\isom C_3[\bar{K}_m]$ for an $m\geq n$ or $n=1$ and $N^-(y)\isom \bar{K}_m[C_3]$ for an $m\geq 1$ which has to be equal to~$1$ by Lemma~\ref{lem_KnC3ImpliesH}, so in each case $N^-(y)\isom C_3[\bar{K}_m]$ for an $m\geq n$.
By symmetry we conclude $m=n$.
Then there is a vertex $z\in N^-(x)\cap N^-(y)$.

\begin{clm}\label{clm_C3KnImpliesHKn1}
$N^+(x)\cap N^+(z)$ is an independent set of cardinality $n$.
\end{clm}

\begin{proof}[Proof of Claim~\ref{clm_C3KnImpliesHKn1}]
This is a direct consequence of the fact that $N^+(z)$ is isomorphic to $C_3[\bar{K}_n]$.
\end{proof}

An immediate consequence of the C-homogeneity of~$D$ is $N^+(x)\cap N^-(z)\neq\es$.

\begin{clm}\label{clm_C3KnImpliesHKn2}
$N^+(x)\cap N^-(z)$ is an independent set of cardinality $n$.
\end{clm}

\begin{proof}[Proof of Claim~\ref{clm_C3KnImpliesHKn2}]
We already know that the set $N^+(x)\cap N^-(z)$ is not empty.
So let us suppose that there is an edge $ab$ with both of its incident vertices in $N^+(x)\cap N^-(z)$.
Then the digraphs $D[z,x,a]$ and $D[z,x,b]$ are isomorphic and hence there is an automorphism of~$D$ mapping the first onto the second one.
As a consequence of Claim~\ref{clm_C3KnImpliesHKn1} both, $a$ and $b$, have to be adjacent to all the vertices in $N^+(x)\cap N^+(z)$.
Hence there is $y'a\in ED$ and $by'\in ED$ for all $y'\in N^+(x)\cap N^+(z)$.
Thus no such automorphism can exist and we conclude that no such edge $ab$ can exist.
Since there are at least $n$ vertices in $N^-(y)$ that lie in $N^+(x)\cap N^-(z)$ and since there are at most $n$ vertices in $N^+(x)$ that are pairwise not adjacent, the assertion follows.
\end{proof}

\begin{clm}\label{clm_C3KnImpliesHKn3}
There is $|N^+(x)\cap N^+(z)|=n=|N^+(x)\cap N^-(y)|$.
\end{clm}

\begin{proof}[Proof of Claim~\ref{clm_C3KnImpliesHKn3}]
This is a direct consequence of the fact that the subdigraph induced by $N^+(x)$ is isomorphic to $C_3[\bar{K}_n]$.
\end{proof}

\begin{clm}\label{clm_C3KnImpliesHKn4}
There is an equivalence relation $\sim$ on $VD$ whose equivalence classes have precisely $n$ independent vertices each and such that $D_\sim$ is isomorphic to $H$ and $D_\sim[\bar{K}_n]$ is isomorphic to $D$.
\end{clm}

\begin{proof}[Proof of Claim~\ref{clm_C3KnImpliesHKn4}]
Let us define a relation $\sim$ via
$$a\sim b:\Lra N^-(a)=N^-(a)\cap N^-(b)=N^-(b).$$
Then $\sim$ is obviously an equivalence relation.

If we consider two of the equivalent classes of~$\sim$, then all of the edges between these two classes must be directed in the same direction and furthermore the digraph induced by these two classes is a complete bipartite digraph.
Hence $D$ induces a C-homogeneous digraph on $D_\sim$ with $D\isom D_\sim[\bar{K}_n]$.

It is a straightforward argument to show that $D\isom H$ if $N^+(x)\isom C_3$.
So if we consider $D_\sim$, then we may instead assume that $N^+(x)\isom C_3$ for all $x\in D_\sim$ and hence we obtain the isomorphism.
\end{proof}

The lemma is a direct consequence of the previous claim.
\end{proof}

\section{The~independent~case}\label{sec_IndCase}

In this section we consider the situation that every out-neighborhood---and hence by the results of Section~\ref{sec_NonIndCase} also every in-neighborhood---is independent.
The first case we classify is when every vertex has in- or out-degree $1$.

\begin{lem}\label{lem_IndependentDegree1}
Let $D$ be a locally finite connected C-homogeneous digraph and let $x\in VD$.
If $N^+(x)$ or $N^-(x)$ consists of precisely one vertex, then $D$ is either an infinite tree or a directed cycle.
\end{lem}

\begin{proof}
By symmetry we may assume that $N^+(x)$ consists of precisely one vertex.
Let $F$ be the subdigraph of~$D$ that is induced by all descendants of~$x$.
Then either this contains a directed ray or a directed cycle.
If $F$ contains a directed cycle, then, by the C-homogeneity, $x$ has to lie on such a cycle, say $C$.
Suppose that a vertex $y$ exists that lies not on~$C$ but has a successor on~$C$.
Let $\alpha$ be an automorphism of~$D$ with $x^\alpha=y$.
Then $C^\alpha\cap C$ contains the successor of~$y$ and hence $y$ has to lie on~$C$ since every vertex of~$C$ has its unique successor on~$C$.
So in this case we conclude that $D$ is a directed cycle.

We now assume that no directed cycle lies in~$F$.
Let $H$ be the digraph that is induced by all ancestors of vertices of~$F$.
As $|N^+(x)|=1$ and as $D$ is connected, $H$ has to be the whole digraph $D$.
Now let us suppose that there is an undirected cycle in~$D$.
Then there has to be a vertex on that cycle that has out-degree at least~$2$ since $F$ is a ray, contrary to the assumption.
Hence $D$ is an infinite tree.
\end{proof}

For the investigation of the C-homogeneous digraphs with bipartite reachability digraph we use the classification of the locally finite C-homogeneous bipartite graphs.
A bipartite graph $G$ (with bipartition $X\cup Y$) is {\em connected-homogeneous bipartite}, or short {\em C-homogeneous bipartite}, if every isomorphism between two isomorphic connected finite subgraphs $A$ and $B$ of~$G$ that preserves the bipartition (that means $VA\cap X$ is mapped onto $VB\cap X$ and $VA\cap Y$ is mapped onto $VB\cap Y$) extends to an automorphism of $G$ that preserves the bipartition.

The next lemma is due to Gray and M\"oller \cite[Lemma 4.3]{GM-CHomDigraphs}, see also \cite[Lemma 5.4]{HH-ConHomDigraphs}, and it underlines our interest in the C-homogeneous bipartite graphs.

\begin{lem}\label{lem_GM4.3}
Let D be a connected C-homogeneous digraph such that $\Delta(D)$ is bipartite.
Then the underlying undirected graph of $\Delta(D)$ is a connected C-homoge\-neous bipartite graph.\qed
\end{lem}

The next result is the classification result of the C-homogeneous graphs.
The proof of Theorem~\ref{thm_GMoThm4.6} which is due to Gray and Möller \cite[Theorem 4.6]{GM-CHomDigraphs} uses the classification of the homogeneous bipartite graphs, see~\cite{GGK}.

\begin{thm}\label{thm_GMoThm4.6}
Let $G$ be a locally finite connected graph.
Then $G$ is C-homoge\-neous bipartite if and only if $G$ is isomorphic to one of the following graphs:
\begin{enumerate}[{\em (i)}]
\item a cycle $C_{2m}$ with $m\geq 2$,
\item an infinite semiregular tree $T_{k,l}$ with $k,l\geq 2$,
\item a complete bipartite digraph $K_{m,n}$ with $m,n\geq 1$,
\item a complement of a perfect matching $CP_k$ with $k\geq 2$.\qed
\end{enumerate}
\end{thm}

In the following we will at first suppose that our digraphs suffices the assumptions of~Theorem~\ref{thm_GMoThm4.6} to obtain partial classification results which will be completed in Section~\ref{sec_Imprimitive}.
Thereafter we shall prove in the lemmas~\ref{lem_NoC3ThenBipReach} and \ref{lem_C3ThenBipReach} that the connected locally finite C-homogeneous digraphs indeed always satisfy these assumptions.

\begin{lem}\label{lem_DeltaBipThenFinite}
Let $D$ be a locally finite connected C-homogeneous digraph such that $N^+(x)$ and $N^-(x)$ are independent sets for all $x\in VD$.
If $\Delta(D)$ is bipartite, then either $\Delta(D)$ is a finite digraph or $C_3$ embeds into~$D$ and $\Delta(D)\isom T_{2,2}$.
\end{lem}

\begin{proof}
Suppose that $\Delta(D)$ is not finite.
Since $D$ is locally finite, we conclude from Theorem~\ref{thm_GMoThm4.6} that $\Delta(D)\isom T_{k,l}$ for integers $k,l\geq 2$.
We distinguish two cases: Either $C_3$ embeds into $D$ or not.
So let us first suppose that $C_3$ does not embed into~$D$.
Let $\Delta_1,\Delta_2$ be two distinct reachability digraphs with non-empty intersection and let us denote with $d_i$ the distance in $\Delta_i$ between vertices of $\Delta_i$.
If $\Delta_1$ and $\Delta_2$ intersect in at most one vertex, let $x,y,z\in VD$ with $xz,yz\in ED$, $x,y,z\in V\Delta_1$, $z\in V\Delta_2$.
Then there is a ray $R$ in $\Delta_2$ starting in $z$ and such that no vertex on~$R$ except for~$z$ is adjacent to~$x$ or~$y$ because none of the out-neighbors of~$z$ is adjacent to~$x$ or~$y$.
Let $a\in V\Delta_1$ with $r:=d_1(a,x)\geq 2$ and $d_1(a,x)<d_1(a,y),d_1(a,z)$.
Then there is a path $P$ outside $B_{r+2}(x)$ from $a$ to~$R$.
Let $P'$ be the (induced) path in $P\cup R$ from $a$ to~$z$.
Then the subdigraphs $P'\cup\{x\}$ and $P'\cup\{y\}$ are isomorphic---we can map $x,z,z_1$ onto $x,z,z_2$ for any two successors of~$z$ and thus we may conclude that no vertex of~$\Delta_2$ except for~$z$ is adjacent to~$x$ or to~$y$.
But there is no automorphism of~$D$ mapping the first onto the second since $d_1(a,x)<d_1(a,y)$.
Thus we have $|\Delta_1\cap\Delta_2|\geq 2$ and hence there are infinitely many vertices in $\Delta_1\cap\Delta_2$ because of the C-homogeneity of~$D$.

If there are two vertices $u,v$ in $\Delta_1\cap \Delta_2$ with minimal distance $d_i(u,v)$ and with $d_i(u,v)\geq 3$ and $d_j(u,v)\geq 2$ ($i\neq j$), then we get a contradiction by two analog paths as before.

So we conclude that for all $u,v\in V\Delta_1\cap V\Delta_2$ with minimal distance in $\Delta_1$ there is $d_1(u,v)=2=d_2(u,v)$.
Now we shall construct a cycle in $\Delta_2$.
Let $x_1$ be the vertex in $\Delta_1$ that is adjacent to both $u$ and $v$ and let $x_2$ be another neighbor of $u$ in $\Delta_1$.
Let $y_1$ and $y_2$ be analog vertices in $\Delta_2$.
Then there is an automorphism of $D$ that fixes $u,y_1$ and $y_2$ and maps $x_1$ onto $x_2$ and vice versa.
Hence $x_1$ has another neighbor $v'$ in $\Delta_1$ that is also a neighbor in $\Delta_2$ of $y_2$.
But as $d_1(v,v')=2$, there is a neighbor $y_3$ of $v$ and $v'$ in $\Delta_2$.
Then the digraph induced by the vertices $u,v,v',y_1,y_2,y_3$ induces a cycle of length $6$ in~$\Delta_2$ which is impossible.

For the last case we suppose that $C_3$ embeds into $D$.
Let $\Delta_1,\Delta_2$ be as above and $xz,yz\in E\Delta_1$ with $z\in \Delta_1\cap\Delta_2$.
Let us suppose that $d^+\geq 3$ or $d^-\geq 3$.
Then we obtain a contradiction similar to the first one, if there is an out-neighbor of~$z$ that is adjacent neither to $x$ nor to~$y$.
So we may assume that there are at least two elements of $N^+(z)$ that are adjacent to~$x$.
As $D$ is C-homogeneous, each two elements of $N^+(z)$ have a common successor, and since $\Delta(D)\isom T_{k,l}$, there is a vertex adjacent to all elements of $N^+(z)$.
This vertex has to be~$x$ by the choice of~$x$.
But by C-homogeneity this also holds for~$y$, so there is a cycle in $\Delta_3$ which is impossible.
So in this case we have $d^+=d^-=2$.
\end{proof}

\begin{lem}\label{lem_DeltaBipThenDirectedCycle2}
Let $D$ be a locally finite connected C-homogeneous digraph with at most one end such that $N^+(x)$ and $N^-(x)$ are independent sets for all $x\in VD$.
If $\Delta(D)$ is bipartite and if the intersection of any two reachability digraphs does not separate each of them, then no reachability digraph separates $D$.
\end{lem}

\begin{proof}
Suppose that there is a reachability digraph $\Delta_1$ that separates $D$.
Let $\Delta_2$ be a reachability digraph with $V\Delta_1\cap V\Delta_2\neq\es$, let $x\in V\Delta_1\cap V\Delta_2$, $y$ be a neighbor of~$x$ in~$\Delta_2$, and let $C_i$, $i=1,2$, be the component of $D-\Delta_i$ that contains $y$ or is adjacent to~$y$ by an edge that lies not in $E\Delta_i$.
If $C_2$ does not contain any vertex of $\Delta_1$, then $C_2\subset C_1$ with $C_2\neq C_1$.
So both $C_1$ and $C_2$ has to be infinite since they are isomorphic.
Thus $D$ has one end in $C_1\cap C_2$ and symmetrically also another one in $(D-C_1)\cap(D-C_2)$ contrary to the assumptions.
So $C_2$ contains a vertex of~$\Delta_1$ and $C_2\not\subset C_1$.
But then, as $\Delta_1\sm V\Delta_2$ is connected, there is another component $C_2'$ of $D-\Delta_2$ that is completely contained in $C_1$ and contains no vertex of~$\Delta_1$.
The component $C_2'$ does not have to be isomorphic to~$C_1$, but since there is a reachability digraph $\Delta_3$ in $C_2'$, we obtain a component $C_3$ of $D-\Delta_3$ with $C_3\subset C_2'$ and so on.
Because the degree of any vertex is finite, there are $m,n$ such that $C_m$ and $C_n$---or $C_2'$ if $m$ or~$n$ is~$2$---lead to an analog contradiction as before.
\end{proof}

The following lemma is the main lemma for the case that there is no isomorphic copy of~$C_3$ in the C-homogeneous digraph.
After its proof, we show in Lemma~\ref{lem_NoC3ThenBipReach} that in the case that the out-neighborhood and the in-neighborhood each are independent sets the connected locally finite C-homogeneous digraphs that contains directed triangles always satisfy the assumptions of Lemma~\ref{lem_DeltaBipThenDirectedCycle}, so that in this case the conclusion of Lemma~\ref{lem_DeltaBipThenDirectedCycle} holds.

\begin{lem}\label{lem_DeltaBipThenDirectedCycle}
Let $D$ be a locally finite connected C-homogeneous digraph that contains no directed triangle and such that $N^+(x)$ and $N^-(x)$ are independent sets for all $x\in VD$.
If $\Delta(D)$ is bipartite, then either $D$ has at least two ends or~$D$ is isomorphic to $C_m[\bar{K}_n$ for an $m\geq 4$, $n\geq 1$.
\end{lem}

\begin{proof}
Let $\Delta(D)$ be bipartite.
We suppose that $D$ contains at most one end and, by Lemma~\ref{lem_IndependentDegree1}, that $d^+,d^-\geq 2$.

\begin{clm}\label{clm_DeltaBipThenDirectedCycle4}
Let $\Delta_1,\Delta_2$ be two reachability digraphs with non-trivial intersection.
Then either their intersection is contained in the same side of the bipartition of~$\Delta_1$ or $\Delta(D)\isom CP_k$ for a $k\geq 3$ and the intersection consists of precisely one unmatched pair in~$CP_k$.
\end{clm}

\begin{proof}[Proof of Claim~\ref{clm_DeltaBipThenDirectedCycle4}]
Suppose that the claim does not hold.
We remember that the reachability digraphs are induced subdigraphs.
So $D$ consists of at least $3$ reachability digraphs.
We also conclude that $\Delta(D)$ cannot be a complete bipartite digraph, so it is either the complement of a perfect matching or a directed cycle.
Let $x,y\in\Delta_1\cap\Delta_2$ be on distinct sides of $\Delta_1$ (and hence also of $\Delta_2$) with minimal distance in $\Delta_1$.
If $\Delta_1\isom C_{2m}$ for an $m\geq 4$, then we choose a minimal path $P$ in $\Delta_2$ from $x$ to~$y$.
Let $x'$ be a neighbor of~$x$ in~$\Delta_1$ and $y_1,y_2$ two neighbors of~$y$ in~$\Delta_1$.
Then by mapping $Pyy_1$ onto $Pyy_2$ we obtain $d_{\Delta_1}(x,y)=m$ and hence the subdigraphs induced by $x'xPyy_1$ and $x'xPyy_2$ are isomorphic paths.
We conclude that also $d_{\Delta_1}(x',y)=m$, a contradiction.
So $\Delta(D)$ is isomorphic to $CP_k$ for a $k\geq 3$.
Then $\Delta_1\cap\Delta_2$ consists of precisely two vertices that are not matched as claimed.
\end{proof}

Since each vertex lies in at most two reachability digraphs, we consider the following two relations:
Let $x\sim y$ for $x,y\in VD$ if $x$ and $y$ {\em lie on the same side} of a reachability digraph, that is, both have the same out-degree and the same in-degree in that reachability digraph and one of these two values is~$0$.
Let $x\approx y$ for $x,y\in VD$ if $x$ and $y$ lie on the same side of two reachability digraphs.

\begin{clm}\label{clm_DeltaBipThenDirectedCycle3}
Let $x,y\in VD$.
Then $x\sim y$ if and only if $x\approx y$.
\end{clm}

\begin{proof}[Proof of Claim~\ref{clm_DeltaBipThenDirectedCycle3}]
Let $x,y\in VD$.
It suffices to prove that $x\sim y$ implies $x\approx y$.
So let us suppose that $x\sim y$ but $x\not\approx y$ and let $\Delta$ be the reachability digraph that contains both vertices, $x$ and $y$, in on the same side.
Since $\Delta(D)$ is finite by Lemma~\ref{lem_DeltaBipThenFinite}, we know from Theorem~\ref{thm_GMoThm4.6} that both sides of the reachability digraph have the same size.
Hence there is a vertex with two successors in distinct reachability digraphs and one with two predecessors in two distinct reachability digraphs.
We conclude by the C-homogeneity that for every vertex each two successors lie in precisely one common reachability digraph and the same holds for each two predecessors.
So we may assume that $x$ and $y$ have distance $2$.
Let $v_1$ be a vertex in the same reachability digraph as $x$ and $y$ that is adjacent to both $x$ and $y$.
By symmetry we may assume that $xv_1,yv_1\in ED$.

The next aim is to show that no reachability digraph separates $D$.
Let us suppose that the converse holds.
By Lemma~\ref{lem_DeltaBipThenDirectedCycle2} each two reachability digraphs that have at least one common vertex, have at least two common vertices.
We conclude from Lemma~\ref{lem_DeltaBipThenDirectedCycle2} and Claim~\ref{clm_DeltaBipThenDirectedCycle4} that the intersection of each two reachability digraphs is contained in one side of the bipartition of each one.
But then Theorem~\ref{thm_GMoThm4.6} implies that $\Delta(D)$ is a cycle of length $2m$ with $m\geq 4$.
So let $a,b$ be two vertices in the same two reachability digraphs $\Gamma_1$ and $\Gamma_2$ with minimal distance.
Then there is a minimal path $P$ between $a$ and $b$ in~$\Gamma_1$.
Let $w_1,w_2$ be the neighbors of $b$ in~$\Gamma_2$, let $u_1$ be the vertex on~$P$ that is adjacent to~$a$, and let $u_2$ be a vertex in~$\Gamma_2$ that is adjacent to~$a$.
Then the paths $u_1Pbw_1$ and $u_1Pbw_2$ are isomorphic and thus there is an automorphism of~$D$ that maps the first onto the second one.
This automorphism has to fix~$a$ and thus the distance in $\Gamma_2$ from $a$ to~$w_1$ is the same as the one from $a$ to~$w_2$.
But since $m\geq 4$, we can also map the path $u_2aPbw_1$ onto $u_2aPbw_2$.
Then also $u_2$ and $w_1$ have the same distance in $\Gamma_2$ as $u_2$ and $w_2$.
But this cannot be true.
Thus we know that no reachability digraph can separate $D$.

Hence we find an (undirected) induced path $R$ from $v_1$ to~$y$ whose only vertices in~$\Delta$ are $v_1$ and $y$ and that does not use the edge $yv_1$.
We may choose $R$ so that the only vertex on~$R$ that is adjacent to~$x$ is $v_1$ by applying the C-homogeneity to an automorphism that maps $D[x,v_1,y]$ onto $D[y,v_1,x]$.
Let $v_3,v_2,y$ be the last three vertices on $R$.
If $v_3\sim y$, then we conclude $v_3\not\sim x$.
But then $yv_1Rv_3$ can be mapped onto $xv_1Rv_3$ by an automorphism of~$D$ and we obtain by $x\sim v_3$ a contradiction.

So $v_3\not\sim y$.
By an analog automorphism to the one above---one that maps $yv_1Rv_3$ onto $xv_1Rv_3$---we obtain that $v_3$ and $x$ have a common neighbor $v_4$.
Let $w_1$ be the neighbor of~$v_1$ on~$R$.
Since $D$ contains no directed triangle, there is an automorphism $\alpha$ of~$D$ that fixes $w_1Rv_2$ and also $v_4x$ pointwise, but with $y^\alpha\neq y$.
But then $y^\alpha$ has to lie in the reachability digraph $\Delta$ as $y$ and $x$ which is impossible as we already saw.
\end{proof}

We conclude from Claim~\ref{clm_DeltaBipThenDirectedCycle3} that $\sim$ and $\approx$ are equivalence relations on $VD$.
Let $\Gamma$ be a digraph on the equivalence classes of~$\sim$ such that there is an edge from one class $X_1$ to another $X_2$ if and only if there are vertices $x_1\in X_1$ and $x_2\in X_2$ with $x_1x_2\in ED$.
By Claim~\ref{clm_DeltaBipThenDirectedCycle3} each vertex of~$\Gamma$ has precisely one successor and one predecessor.
It is a straightforward argument that $\Gamma$ is a C-homogeneous digraph.
Since $D$ has at most one end, $\Gamma$ must be a directed cycle $C_n$ for an $n\geq 3$ by Lemma~\ref{lem_IndependentDegree1}.

\medskip

It remains to show that the inverse images of any edge of~$\Gamma$, that is the subdigraph of~$D$ induced by the equivalence classes that are incident with that particular edge of~$\Gamma$, and that is precisely one reachability digraph, induces a complete bipartite digraph.
Let $V_1,\ldots, V_n$ denote the equivalence classes such that $V_iV_{i+1}\in E\Gamma$ for $i<n$ and $V_nV_1\in E\Gamma$ with $n\geq 4$.

It follows from Lemma~\ref{lem_GM4.3} and Theorem~\ref{thm_GMoThm4.6} that $\Delta(D)$ is either a semi-regular tree $T_{k,l}$, a cycle $C_{2m}$, the complement of a perfect matching $CP_k$, or a complete bipartite digraph $K_{k,l}$.
To prove the lemma, we have to show that none of the first three cases can occur where the first one was already excluded by Lemma~\ref{lem_DeltaBipThenFinite}.

Let us suppose that $\Delta(D)\isom C_{2m}$ for an $m\geq 4$.
Let $x\in V_1$ and let $a,b$ be its successors.
Let $P$ be a shortest $a$-$b$-path in $\Delta_2$, the subdigraph induced by $V_2$ and $V_3$, and let $P^\circ:=P-b$.
Let $P'$ be a path of the same length as $P^\circ$ in $\Delta_2$ that starts in~$a$ and is except for $a$ distinct from $P$.
By mapping $xP^\circ$ onto $xP'$, we obtain $d_{\Delta_2}(a,b)=m$.
But then the same holds for the other predecessor $y\neq x$ of~$a$ and thus $y$ also has to be adjacent to~$b$ and hence $m=2$, a contradiction.

Let us now suppose that $\Delta(D)\isom CP_k$ for a $k\geq 3$.
Let $x\in V_1$.
Then there exists a unique vertex in $V_2$ that is not adjacent to~$x$ and this vertex itself has a unique vertex $y\in V_3$ it is not adjacent to.
Now let $X$ be the digraph $D[(V_3\sm\{y\})\cup P]$ where $P$ denotes a path that consists of one vertex from each $V_i$, $i\geq 4$ and of~$x$ such that the vertex in $V_4$ is the only vertex incident with all of $V_3$ but $y$.
Let $x'$ be another vertex of $V_1$ that is adjacent to the predecessor of~$x$ on~$P$ and let $Y$ be the digraph $D[(VX\sm\{x\})\cup\{x'\}]$.
Then $X$ and $Y$ are isomorphic subdigraphs of~$D$ but there is no automorphism of~$D$ mapping the first onto the second one since for $x$ and $y$ there is a unique vertex in $V_2$ that is not adjacent to both, but for $x'$ and $y$ there is no such vertex.
Hence $\Delta(D)\not\isom CP_k$ and thus we conclude from Theorem~\ref{thm_GMoThm4.6}, since we excluded all other cases, that $\Delta(D)$ is a complete bipartite digraph.
As all equivalence classes have the same size, $\Delta(D)\isom K_{k,k}$ for a $k\geq 1$.
\end{proof}

The following proposition is similar to a result by Malni\v{c} et.al., see \cite[Proposition~3.2]{MMMSTZ}.
Since we apply it in a situation where its original assumptions need not to be satisfied, we formulated the result with different assumptions.
But the general idea of the proof of Proposition~\ref{prop_MMMSTZ3.2MoreGeneral} is quite similar to the one of the proof of \cite[Proposition 3.2]{MMMSTZ}.
Because our assumptions are to handle differently, we prove it here.

\begin{prop}\label{prop_MMMSTZ3.2MoreGeneral}
Let $D$ be a connected C-homogeneous digraph such that in-degree and out-degree of any vertex are at most a fixed integer $d$ and such that both $N^+(x)$ and $N^-(x)$ are independent sets.
Let $\Gamma=\Aut(D)$, $xy\in ED$ and $\Omega\subseteq N^+(x)$ with $|\Omega|=d$ such that $H=\Gamma_{xy}$ fixes $\Omega$ setwise but stabilizes no vertex of~$\Omega$.

Then there is no alternating walk whose first edge is $xy$ and which ends at a vertex of~$\Omega$.
\end{prop}

\begin{proof}
Since $D$ is C-homogeneous, the group $H$ acts on $\Omega$ like $S_\Omega$, i.e.\ $H^\Omega\isom S_\Omega$.
Let $P$ be an alternating walk with initial edge $xy$.
Suppose that $H^\Omega_P=H^\Omega$.
Let $e\in ED$ such that $Pe$ determines an alternating walk, and let $z$ be the vertex incident with $e$ but distinct from the end vertex of~$P$.
Then there are at most $d-1$ vertices in $\{z^\alpha\mid\alpha\in H_P\}$.
Since $H^\Omega=H^\Omega_P$, either we have $d=2$ or we have $|H^\Omega : H^\Omega_z|<d$.
Let us first assume that $d\neq 2$.
Then, by Proposition~\ref{prop_Wielandt}, $H^\Omega_z$ is either $H^\Omega$ or isomorphic to $A_\Omega$.
We shall now show that the latter case cannot occur.
So suppose that $H^\Omega_z\isom A_\Omega$.
Then $H_z$ acts transitively on~$\Omega$ but there is no automorphism fixing $|\Omega|-2$ elements and switching the other two.
Since $D$ is C-homogeneous and $\Omega$ is independent, this is impossible.
Hence $H^\Omega_z=H^\Omega$ and thus no vertex of~$\Omega$ is fixed by $H_z$.
So let us now assume that $d=2$.
But in this case we immediately deduce from the fact that the orbit of~$z$ under $H$ contains only~$z$, that $H=H_z$.
So we conclude in each case that no vertex of $\Omega$ can lie on an alternating walk.
\end{proof}

\begin{lem}\label{lem_NoC3ThenBipReach}
Let $D$ be a connected locally finite C-homogeneous digraph such that $N^+(x)$ and $N^-(x)$ are independent sets for all $x\in VD$ and assume that $D$ contains no directed triangle.
Then the reachability relation of $D$ is not universal.
\end{lem}

\begin{proof}
Let $xy\in ED$.
By symmetry we may assume that $d^+(x)\geq d^-(x)$.
Let $\Omega=N^+(y)$.
By applying Proposition~\ref{prop_MMMSTZ3.2MoreGeneral} we conclude that no vertex of~$\Omega$ lies on an alternating path that starts with the edge $xy$ and thus that the reachability relation of~$D$ cannot be universal.
\end{proof}

\begin{lem}\label{lem_NoKnButC3Degree}
Let $D$ be a connected locally finite C-homogeneous digraph such that $N^+(x)$ and $N^-(x)$ are independent sets for all $x\in VD$.
If $C_3$ embeds into $D$, then
we have $d^+(x)=d^-(x)$.
\end{lem}

\begin{proof}
Let $y\in N^+(x)$ and $z\in N^-(x)\cap N^+(y)$.
The number $n_1$ of directed triangles that contain $xy$ is equal to the number of all $2$-arcs from $y$ to~$x$.
By C-homogeneity this is the same as the number of all $2$-arcs from $x$ to~$z$ which is again equal to the number $n_2$ of all directed triangles that contain $zx$.
Let $n_3$ denote the number of all directed triangles that contain~$x$.
Then we conclude from the C-homogeneity
$$|N^+(x)| n_1=n_3=|N^-(x)| n_2.$$
Since $n_1=n_2$, the claim follows.
\end{proof}

\begin{lem}\label{lem_NoKnButC3NumberOfC3}
Let $D$ be a locally finite C-homogeneous digraph that contains a directed triangle.
Then for every edge $xy\in ED$ the number of directed cycles that contains $xy$ is either $1$ or at least $(d^+-1)$.
\end{lem}

\begin{proof}
Let $d_1$ be the number of elements of $N^+(y)$ that lie on a common directed triangle with $xy$ and let $d_2$ be the number of elements of $N^+(y)$ for which this is not the case.
Then $d=d_1+d_2$ where $d:=d^+$ which is the same as~$d^-$ by Lemma~\ref{lem_NoKnButC3Degree}.
Let $\Omega_1$ be the set of all vertices of $N^+(y)$ that lie on a common directed triangle with $xy$ and let $\Omega_2=N^+(y)\sm \Omega_1$.
Let $\Omega_3:=N^+(x)\sm\{y\}$.
We consider the action of $H:=\Aut(D)_{xy}$ on~$\Omega_3$.
Because $N^+(x)$ is an independent set, $H$ acts on $\Omega_3$ like $A_{\Omega_3}$.
For $z\in\Omega_i$, $i=1,2$, we have $|H:H_z|=d_i<d^+-1=|\Omega_3|$.
Thus and by Proposition~\ref{prop_Wielandt}, $H_z$ acts on $\Omega_3$ either like $S_{\Omega_3}$ or like $S_{\Omega_3}$.
Let us first consider the second case.
By a similar argument as in Proposition~\ref{prop_MMMSTZ3.2MoreGeneral}, we know that $|\Omega_3|=2$.
But then $d^+=3$ and the assertion trivially holds.
So we assume that $H_z$ acts on $\Omega_3$ like $S_{\Omega_3}$.
In that case either no vertex of $N^+(y)$ lies in $N^+(y')$ for any $y\neq y'\in N^+(x)$ or every vertex of $N^+(y)$ lies in $N^+(y')$ for every $y\neq y'\in N^+(x)$.
We conclude that each edge lies either on precisely one or on $d$ distinct directed triangles.
\end{proof}

The following lemma is the main lemma for the case that the C-homogeneous digraph contains a directed triangle.
The assumption that $\Delta(D)$ is bipartite in this case shall be verified in Lemma~\ref{lem_C3ThenBipReach} and the case (iv) of the conclusions shall be investigated in Section~\ref{sec_Imprimitive}.

\begin{lem}\label{lem_DeltaBip+C3ThenDirectedCycle}
Let $D$ be a locally finite connected C-homogeneous digraph that contains an isomorphic copy of~$C_3$ and such that $N^+(x)$ and $N^-(x)$ are independent sets for all $x\in VD$.
If $\Delta(D)$ is bipartite, then one of the following cases holds.
\begin{enumerate}[{\em (i)}]
\item The digraph $D$ has at least two ends.
\item The reachability digraph $\Delta(D)$ is isomorphic to $K_{k,k}$ for a $k\geq 3$ and $D$ is isomorphic to $C_3[\bar{K}_k]$.
\item The reachability digraph $\Delta(D)$ is isomorphic to $CP_k$ for a $k\in\nat$ and $D$ is isomorphic to a $Y_k$ for a $k\geq 3$.
\item The reachability digraph $\Delta(D)$ is isomorphic to $C_{2m}$ for an $m\geq 4$ or to~$T_{2,2}$.
\end{enumerate}
\end{lem}

\begin{proof}
Let us assume that the digraph $D$ has at most one end, that case (iv) does not occur, and that $d^+,d^-\geq 2$ by Lemma~\ref{lem_IndependentDegree1}.
By Theorem~\ref{thm_GMoThm4.6}, we know that $\Delta(D)$ is either a semi-regular tree---which is impossible in our situation because of Lemma~\ref{lem_DeltaBipThenFinite}---, a complete bipartite digraph, a $CP_k$, or a cycle $C_{2m}$---which we also excluded.
Let us first assume that $\Delta(D)$ is a complete bipartite digraph $K_{k,l}$ for $k,l\in\nat$ but not $K_{2,2}$.
By Lemma~\ref{lem_NoKnButC3Degree}, we know that $k=l$.
If, for two reachability digraphs $\Delta,\Delta'$, there is $|\Delta\cap\Delta'|\geq 2$, then it is a direct consequence of the C-homogeneity that $\Delta\cap\Delta'$ is a complete side of each of $\Delta,\Delta'$.
Thus it is---like in the proof of Lemma~\ref{lem_DeltaBipThenDirectedCycle}---a~direct consequence that (ii) holds in this case.
So let us suppose that $\Delta\cap\Delta'$ has cardinality~$1$.
If an edge lies on more than one directed triangle, then we know from Lemma~\ref{lem_NoKnButC3NumberOfC3} that it lies on at least $k-1$ distinct such triangles.
But then, the intersection $\Delta\cap\Delta'$ has to contain at least $k-1$ elements which is a contradiction.
So every edge lies on a uniquely determined directed cycle of length~$3$.

\begin{clm}\label{clm_DeltaBip+C3ThenDirectedCycle2}
For every four distinct reachability digraphs $\Delta_1,\Delta_2,\Delta_3,\Delta_4$ such that $\Delta_i\cap\Delta_{i+1}$ ($i=1,2,3$) is not empty and such that $(\Delta_{i-1}\cup\Delta_{i+1})\cap\Delta_i$ lies on the same side of $\Delta_i$ for $i=2,3$, $\Delta_1\cap\Delta_4$ is not empty, too, and its intersection lies on the same side of $\Delta_4$ as $\Delta_3\cap\Delta_4$.
\end{clm}

\begin{proof}[Proof of Claim~\ref{clm_DeltaBip+C3ThenDirectedCycle2}]
Let us suppose that $\Delta_1\cap\Delta_4$ is empty.
Since every edge lies on a directed triangle, there has to be a vertex $x$ with successors in $\Delta_1$ and $\Delta_2$.
Let $y$ be its sucessor in $\Delta_1$, $z$ be its successor in $\Delta_2$ and let $a$ be the vertex in $\Delta_2\cap\Delta_3$, $b$ the vertex in $\Delta_3\cap\Delta_4$.
Then there is no automorphism of~$D$ that maps $a$ to~$b$ and fixes all of $x,y,z$, because $\Delta_1\cap\Delta_2\neq\es$ but $\Delta_1\cap\Delta_4=\es$.
Hence we proved the claim.
\end{proof}

Now we are able to show that the whole situation cannot occur.
Let $x,y$ be two vertices on the same side of a reachability digraph such that their out-degree is~$0$.
Let $a,b$ be successors of $x,y$, respectively, such that they lie in a common reachability digraph.
As $k\geq 3$ and as every edge lies on precisely one copy of~$C_3$, there is a successor $c$ of $a$ and $b$ such that neither $D[x,a,c]$ nor $D[y,b,c]$ are triangles.
Furthermore, there exists a predecessor $z$ of~$b$ such that $z$ and $c$ are not adjacent.
The vertices $a$ and $z$ cannot be adjacent because otherwise $y$ and $x$ have to lie in two common reachability digraphs which we supposed to be false.
Then $D[x,a,c,b,y]$ and $D[x,a,c,b,z]$ are isomorphic, but there is no automorphism of~$D$ that maps one onto the other just by fixing all of $x,a,c,b$.
Thus we showed that there are no two reachability digraphs whose intersection consists of precisely one vertex.
This finishes the case $\Delta(D)\isom K_{k,l}$.

\medskip

The next and final situation which we consider is $\Delta(D)\isom CP_k$ for a $k\geq 4$.
Let $\Delta_1,\Delta_2$ be two distinct reachability digraphs of~$D$ with non-trivial intersection.
We prove that $|\Delta_1\cap\Delta_2|\geq 2$.
So suppose that $|\Delta_1\cap\Delta_2|=1$.
Let $b\in VD$ and let $a,c$ be two predecessors of~$b$.
Let $x,y$ be two predecessors of~$a$ and let $v$ ($w$) be a vertex such that $av,vx$ ($aw,wy$, respectively) lie in~$ED$ and such that $cw$ but not $cv$ lies in~$ED$.
Then the digraphs $D[a,b,c,x]$ and $D[a,b,c,y]$ are isomorphic but there is no automorphism of~$D$ that maps them onto each other because such an automorphism has to map $v$ onto~$w$ but $w$ is adjacent to $c$ and $v$ is not.

\begin{clm}\label{clm_DeltaBip+C3ThenDirectedCycle1}
The set $\Delta_1\cap\Delta_2$ is one whole side of each of $\Delta_1,\Delta_2$.
\end{clm}

\begin{proof}[Proof of Claim~\ref{clm_DeltaBip+C3ThenDirectedCycle1}]
Let us first suppose that $\Delta_1\cap\Delta_2$ is not contained in any of the sides of $\Delta_1$.
Then $\Delta_1\cap \Delta_2$ consists of precisely two vertices that are adjacent in the bipartite complement of~$\Delta_1$.
Let us consider the subdigraph of~$\Delta_i$ with vertices $a,b,c,d$, with edges $ba,bc,dc$ such that $a,d\in\Delta_1\cap\Delta_2$.
Let $x,y$ be two predecessors of $d$ in~$\Delta_j$ with $i\neq j$ and let $z$ be the neighbor of~$x$ in the bipartite complement of~$\Delta_j$.
Since each edge lies on a directed triangle, we may assume that $b,a,z$ form such a triangle and, since $k\geq 4$, we also may assume that $c$ and $y$ do not lie in a common reachability digraph.
Then neither $c$ nor $y$ lies in a common reachability digraph with $b$ and~$z$.
So each of the subdigraphs $D[a,b,c,d,x]$ and $D[a,b,c,d,y]$ contains precisely $4$ edges and they are isomorphic to each other.
Hence there is an automorphism $\alpha$ that fixes each of $a,b,c,d$ and maps $x$ onto~$y$ which is impossible because $y$ and $b$ do not lie in any common reachability digraph in contrast to~$x$ and~$b$.
Thus we proved that $\Delta_1\cap \Delta_2$ is contained in one side of $\Delta_i$, $i=1,2$.

The C-homogeneity directly implies that $\Delta_1\cap\Delta_2$ is a whole side of~$\Delta_i$, $i=1,2$.
Thus we proved the claim.
\end{proof}

We shall now show that $D\isom Y_k$.
Let $\overline{D}$ denote the tripartite complement of~$D$
Since $\Delta(D)\isom CP_k$, the digraph $\overline{D}$ is a union of directed cycles.
We want to show that every component of~$\overline{D}$ is a directed cycle of length~$3$.
So let us suppose that this is not the case.
Then there are $x,y\in V_1$ that lie on a common directed cycle of length at least~$6$ and have distance $3$ on that cycle in~$\overline{D}$.
Since $k\geq 3$, there is a vertex $a\in V_2$ that is adjacent to both $x$ and $y$.
We conclude that for every vertex $z\in V_1$, distinct from $x$, we have that $x$ and $z$ lie on a common cycle and have distance $3$ on that cycle.
It is a direct consequence that $k=3$ and $\overline{D}\isom C_9$.
But then there are edges of~$D$ that lie on precisely one copy of~$C_3$ and some lie on two copies which contradicts the C-homogeneity.
Hence we have $D\isom Y_k$.
\end{proof}

\begin{lem}\label{lem_C3ThenBipReach}
Let $D$ be a connected locally finite C-homogeneous digraph that contains a directed triangle.
Furthermore, assume that $N^+(x)$ and $N^-(x)$ are independent sets for all $x\in VD$.
Then the reachability relation of~$D$ is not universal.
\end{lem}

\begin{proof}
Let $d=d^+$.
By Lemma~\ref{lem_NoKnButC3Degree} we have $d=d^-$.
Suppose that the reachability relation of~$D$ is universal.
Let $D_1$ be the digraph depicted in Figure~\ref{pic_D1}.

\begin{figure}[h]
\begin{center}
\includegraphics[width=.20\textwidth]{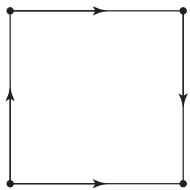}
\caption[Figure 1]{The digraph $D_1$}\label{pic_D1}
\end{center}\end{figure}

\begin{clm}\label{clm_C3ThenBipReach1}
$D$ contains an isomorphic image of~$D_1$.
\end{clm}

\begin{proof}[Proof of Claim~\ref{clm_C3ThenBipReach1}]
Since the reachability relation of~$D$ is universal, there is an induced cycle such that for an edge $xy$ on that cycle the other path between $x$ and $y$ is an alternating path but the such that the whole cycle is not alternating.
Such a cycle shows that the reachability relation is not universal.
To show that such a cycle exists, suppose that it is not the case.
We choose a counterexample $C$ with minimal length.
Since there is always a cycle that shows that the reachability relation is not universal, we may assume that $C$ is not induced.
So there is a chord in~$C$ and hence one of the smaller cycles is a counterexample of smaller length, contrary to the assumption.
Thus a cycle as described exists.

Let us first assume that such a cycle $C$ has odd length.
Then it has length at least~$5$.
By symmetry we may assume that for the edge $xy$ described above we have $d^-_C(x)=1$ and $d^-_C(y)=2$.
Let $z$ be the other vertex in $N^-(y)\cap VC$.
Then there is an automorphism $\alpha$ of~$D$ that maps $C-x$ onto $C-y$.
The digraph $D[x,y,z,x^\alpha]$ is isomorphic to~$D_1$ because $N^-(x)$ and $N^+(x)$ are independent sets.

Let us now consider the case that $C$ is an induced cycle of even length and let $xy$ be again the above described edge.
Let $a\neq y$ be the vertex on~$C$ adjacent to~$x$, let $b\neq x$ be the vertex on~$C$ adjacent to~$y$, let $P_C$ be the path on~$C$ between $a$ and $b$ that contains neither $x$ nor $y$, and let $P_C^-$ denote the path inverse to~$P_C$.
Since $C$ has odd length, we have $P_C\isom P_C^-$.
Then we can map $xP_C$ onto $yP_C^-$ by an automorphism of~$D$ and obtain an induced subdigraph isomorphic to~$D_1$ by the two paths of length $2$ between $a$ and $y$.
\end{proof}

\begin{clm}\label{clm_C3ThenBipReach2}
There is $|N^+(y)\cap N^-(x)|=1$ for all edges $xy\in ED$.
\end{clm}

\begin{proof}[Proof of Claim~\ref{clm_C3ThenBipReach2}]
By Lemma~\ref{lem_NoKnButC3NumberOfC3} we know that either $|N^+(y)\cap N^-(x)|=1$ or $|N^+(y)\cap N^-(x)|\geq d-1$.
So it suffices to prove that $N^+(y)\sm N^-(x)$ contains at least two vertices.
Let $u,v,a,b$ be the four vertices of the digraph $D_1$ such that $u$ has the two predecessors $a$ and $b$.
Since there is an automorphism of~$D$ that fixes $u$ and maps $a$ onto $b$ and vice versa, there is a directed path of length $2$ from $a$ to~$b$ and one from $b$ to~$a$.
Since $N^+(v)$ and $N^-(v)$ are both independent sets and the same holds for $v'$, the image of~$v$ under the described automorphism, there is no edge between $v$ and $v'$.
We may assume by symmetry that $va$ is an edge in~$D$.
Then both $v'$ and $u$ are vertices in $N^+(a)$ that do not lie on a common directed triangle with $va$, so we conclude $N^+(a)\cap N^-(v)$ contains precisely one vertex.
\end{proof}

Let $D_2$ be the digraph shown in Figure~\ref{pic_D2}.

\begin{figure}[h]
\begin{center}
\includegraphics[width=.50\textwidth]{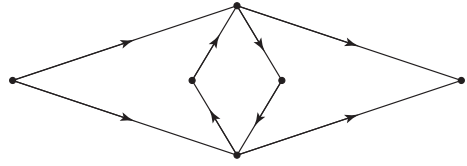}
\caption[Figure 1]{The digraph $D_2$}\label{pic_D2}
\end{center}\end{figure}

\begin{clm}\label{clm_C3ThenBipReach3}
There is an isomorphic image of $D_2$ in~$D$.
\end{clm}

\begin{proof}[Proof of Claim~\ref{clm_C3ThenBipReach3}]
By an analog argument as in the proof of Claim~\ref{clm_C3ThenBipReach2}, we immediately see that either $D_2$ is an induced subgraph or  $D_2$ together with an edge from the vertex on the right hand to the vertex on the left hand is an induced subgraph of~$D$.
But the latter cannot be the case since the additional edge would lie on at least two copies of~$C_3$, contrary to Claim~\ref{clm_C3ThenBipReach2}.
\end{proof}

Now let $D'$ be an isomorphic copy of~$D_2$ in $D$.
Let $x$ be the vertex on the left, $y$ the one on the right and $a,b,u,v$ the vertices of the cycle such that $x$ and $y$ are adjacent to $a$ and $u$.
Since $C_3$ embeds into~$D$, there is a vertex $a'\in N^+(a)\cap N^-(x)$.
Then $a'$ is adjacent neither to~$b$, nor to~$v$, nor to~$y$, since the only directed triangle that contains $aa'$ is $D[x,a,a']$ and since directed cycles are the only cycles of length $3$ that embed into~$D$.
Then there is an automorphism of~$D$ that fixes $a',x$, and $u$, and maps $v$ onto $y$.
This automorphism also has to fix $a$, since it fixes together with $x$ and $a'$ the unique vertex in the directed triangle that contains the edge $a'x$.
Hence such an automorphism cannot exist.
\end{proof}

\section{An imprimitive case}\label{sec_Imprimitive}

In this section we investigate the following situation.
Let $D$ be a C-homogeneous digraph that contains directed triangles of length $3$ and whose reachability digraph is either $T_{2,2}$ or $C_{2m}$ for an $m\geq 2$.
There exists a well-known such digraph \cite{GM-CHomDigraphs}, the digraph $T(2)$ that was defined in the introduction.
This digraph has infinitely many ends.
But although we are interested only in digraphs with at most one end, this particular digraph turns out to be very important in this case.
We shall show that every digraph with the above described properties and with at most one end is a homomorphic image of $T(2)$ in a very particular way.

\begin{thm}\label{thm_ImprimitiveCase}
The following two assertions are equivalent for any locally finite connected digraph~$D$.
\begin{enumerate}[{\em (i)}]
\item The digraph $D$ is C-homogeneous, contains directed cycles of length $3$, and its reachability digraph is either $T_{2,2}$ or $C_{2m}$ for an $m\geq 2$.
\item There exists a subgroup $H$ of $\Aut(T(2))$ acting transitively on $VT(2)$ and an $H$-invariant equivalence relation $\sim$ on $VT(2)$ such that $T(2)_{\sim}$ is isomorphic to $D$.
\end{enumerate}
Furthermore, the digraph has at most one end if and only if each equivalence class consists of more than one element.

In the situation {\em (ii)} we may always choose $H$ to be the whole automorphism group of~$D$.
\end{thm}

\begin{proof}
First, let us assume that (ii) holds.
We choose $H$ so that it is maximal such that $\sim$ is $H$-invariant.

\begin{clm}\label{clm_ImprimitiveCase1}
The stabilizer $H_x$ of any vertex $x$ of~$T(2)$ has order $2$.
\end{clm}

\begin{proof}[Proof of Claim~\ref{clm_ImprimitiveCase1}]
There are two possibilities for an element of~$H_x$.
Either it is fixes both directed triangles that meet $x$ or it changes the two triangles.
As every isomorphism between two directed triangles of~$T(2)$ extends uniquely to an automorphism of~$T(2)$, $|H_x|\leq 2$.
So we have to prove that the element $\alpha\neq id$ of $\Aut(T(2))_x$ is also contained in~$H$.
Because $H$ is the maximal subgroup of~$\Aut(T(2))$ such that $\sim$ is $H$-invariant, we have to prove that $\alpha$ maps one equivalence class onto another.
Let $y,z$ be the two predecessors of $x$.
If they lie in the same equivalence class, then this is fixed by~$\alpha$ and the same holds for the equivalence class that contains both successors of~$x$ and this extends to all equivalence classes because of the transitivity of~$H$ on $VT(2)$.
So $y$ and $z$ lie in distinct equivalence classes.
But then $\alpha$ maps the equivalence class of~$y$ onto the one of~$z$ and vice versa, and the same holds for the two equivalence classes of the two successors of~$x$.
By induction on $d(x,a)$ for any vertex $a$ of~$T(2)$ its equivalence class is mapped onto the one of the unique vertex $b$ with $d(x,b)=d(x,a)$ and for which the shortest path from $x$ to~$b$ is isomorphic to the shortest path from $x$ to~$a$.
So $\alpha$ acts on all equivalence classes and $\sim$ is $\alpha$-invariant.
\end{proof}

To show that $D\isom T(2)_\sim$ is C-homogeneous, let $A,B$ be isomorphic induced connected subdigraphs of~$D$ and $\varphi:A\to B$ be an isomorphism.
Then there are induced subdigraphs $A',B'$ of~$T(2)$ with $A\isom A'$, $B\isom B'$ and such that the equivalence classes of the vertices of~$A$ (of~$B$) are the vertices of $T(2)_\sim$ that induce the digraph $A$ (the digraph $B$, respectively).
Let $\varphi_0$ be the isomorphism $A'\to B'$ that maps the equivalence class of $x\in VA$ to the equivalence class of~$x^\varphi$.
We may assume that $A$ contains an edge $xy$.
Let $u,v$ be vertices in~$T(2)$ such that $uv\in ET(2)$ and such that the equivalence class of $u, v$ is $x, y$, respectively.
Since $H_{x^\varphi}$ has order $2$ by Claim~\ref{clm_ImprimitiveCase1}, there is an automorphism $\alpha\in H$ with $u^\alpha=u^{\varphi_0}$ and with $v^\alpha=v^{\varphi_0}$.
But then the claim immediately implies $A'^\alpha=B'$ and that the canonical image $\alpha'$ of~$\alpha$ maps $A$ onto $B$ like $\varphi$.
Furthermore, $\alpha'$ is an automorphism of~$D$ because $\alpha\in H$.

For the other direction, let $D$ fulfill the assumptions of (i).
Let $\pi$ be the map $T(2)\to T(2)_\sim$ that maps $x\in VT(2)$ onto its equivalence class.
We may assume that $D$ is not isomorphic to~$C_3$.
Let $xy\in ED$, $ab\in ET(2)$.
For every vertex $u$ in~$T(2)$ there exists a unique shortest path $P_1$ from $a$ to~$u$.
In $D$ there are precisely two isomorphic such paths with the property that no two endvertices of any subpath of length $2$ are adjacent.
If the second vertex of the path $P_1$ is~$b$ or is adjacent to~$b$, then let $P_2$ that one of the above described paths in~$D$ whose second vertex is~$y$ or is adjacent to~$y$, and in the other case for $P_1$ let $P_2$ be also the other one in~$D$.
Let $u_D$ denote the last vertex of~$P_2$.

We are now able to define the equivalence relation.
Let $u\sim v$ for two vertices $u,v\in VT(2)$ if $u_D=v_D$.
This is obviously an equivalence relation.
It remains to show that it is $\Aut(T(2))$-invariant.
So let $u$ and $v$ be arbitrary vertices of~$T(2)$ and let $\psi$ be an automorphism of~$T(2)$ with $u^\psi=v$.
We have to show that the equivalence class of~$u$ is mapped onto the one of~$v$.
So let $w\sim u$.
It suffices to consider the case where the shortest path from $u$ to $w$ does not contain any other vertex of the equivalence class that contains $u$.
Let $P$ be the shortest path from $u$ to~$w$.
We look at the paths $P^\pi$ and $(P^\psi)^\pi$.
The path $P^\pi$ starts and ends at the same vertex.
We can map $Q^{\varphi_0}$ for every subpath $Q$ of~$P$ that starts in~$u$ onto $(Q^\psi)^\pi$ inductively, because on the one hand $D$ is C-homogeneous and on the other hand for such a $Q$ its succeeding vertex is uniquely determined in $D$ by the two digraphs $Q^\pi$ and $(Q^\psi)^\pi$.
So we conclude that also $(P^\psi)^\pi$ has the same endvertices.
But then $u^\psi$ and $w^\psi$ have to be equivalent.
It is an immediate consequence that this holds also for any $z\sim u$.

The only remaining part is to show is the additional claim on the multi-ended digraphs which is a direct consequence of~\cite[Theorem~7.1]{GM-CHomDigraphs}.
\end{proof}

Figure~\ref{pic_OneEnded} shows two C-homogeneous digraphs that arise as factor digraphs in Theorem~\ref{thm_ImprimitiveCase} one of which is finite and the other being infinite and one-ended.
In the finite digraph every reachability digraph, which is isomorphic to~$C_{10}$, is drawn in a different shade of gray.
The reachability digraphs of the infinite digraph are the cycles of length~$6$.

\begin{figure}[h]
\begin{center}
\hfill
\includegraphics[width=.40\textwidth]{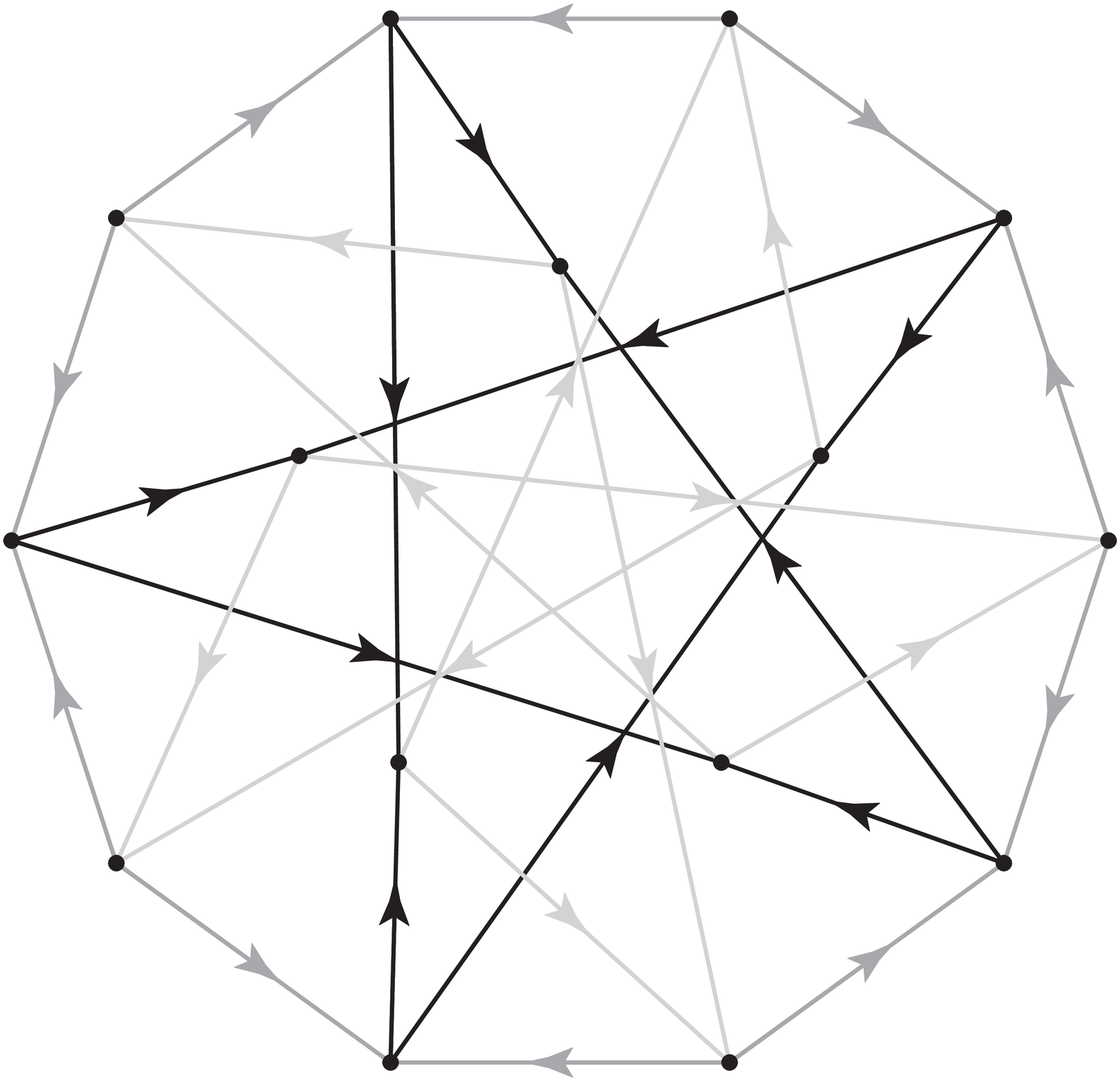}
\hfill
\includegraphics[width=.40\textwidth]{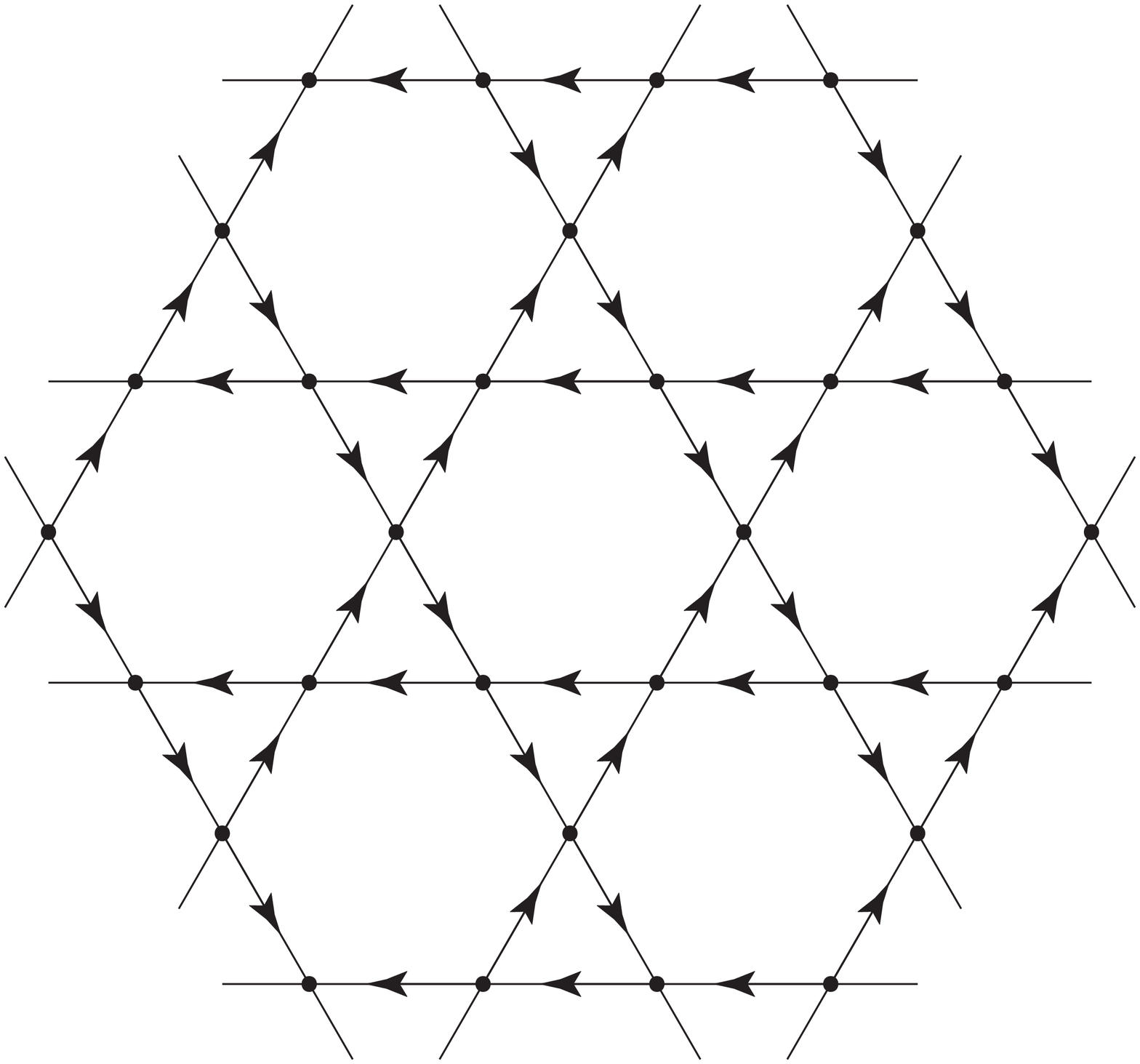}
\hfill $\left.\right.$
\caption[Figure 1]{A finite and an infinite one-ended C-homogeneous digraph}\label{pic_Z522}\label{pic_OneEnded}
\end{center}\end{figure}

\section{The main theorem}

Let us now state our main result.
We shall prove it by applying all the results of the previous sections.

\begin{thm}\label{thm_Main1}\label{thm_Main2}
Let $D$ be a connected digraph with at most one end.
Then $D$ is C-homoge\-neous if and only if one of the following cases holds.
\begin{enumerate}[{\em (i)}]
\item $D\isom C_m[\bar{K}_n]$ for integers $m\geq 3,n\geq 1$;
\item $D\isom H[\bar{K}_n]$ for an integer $n\geq 1$;
\item $D\isom Y_k$ for an integer $k\geq 3$;
\item there exists a non-trivial $\Aut(T(2))$-invariant equivalence relation $\sim$ on $VT(2)$ such that $D\isom T(2)_\sim$.
\end{enumerate}
\end{thm}

\begin{proof}
Let $D$ be a connected locally finite C-homoge\-neous digraph with at most one end.
If the out-neighborhood (or symmetrically the in-neighborhood) of any vertex of~$D$ is not independent, then we conclude from Theorem~\ref{thm_HomDigraphs}, Lemma~\ref{lem_NoH}, Lemma~\ref{lem_NoC4}, Lemma~\ref{lem_KnC3ImpliesH}, and Lemma~\ref{lem_C3KnImpliesHKn} that $D$ is finite and isomorphic to an $H[\bar{K}_n]$ for an $n\geq 1$.
So we may assume that the out-neighborhood of every vertex is independent.
Since $D$ is in particular $1$-arc transitive, we conclude from Proposition~\ref{prop_CPW}, Lemma~\ref{lem_NoC3ThenBipReach}, and Lemma~\ref{lem_C3ThenBipReach} that the reachability digraph of~$D$ is bipartite.
Thus one direction of the theorem follows directly from Lemma~\ref{lem_DeltaBipThenDirectedCycle}, Lemma~\ref{lem_DeltaBip+C3ThenDirectedCycle}, and Theorem~\ref{thm_ImprimitiveCase}.

\medskip

To prove the remaining part of Therorem~\ref{thm_Main1} it suffices to prove that the digraphs $Y_k$ are C-homogeneous because it is an easy consequence of the fact that $H$ is homogeneous, that $H[\bar{K}_n]$ is C-homogeneous.
Furthermore, an immediate consequence of the fact that $K_{n,n}$ is a homogeneous bipartite graph is that $C_m[\bar{K}_n]$ is C-homogeneous and that the graphs in part (iv) are C-homogeneous was already proved in Theorem~\ref{thm_ImprimitiveCase}.
To prove that the digraphs $Y_k$ with $k\geq 3$ are C-homogeneous, let $A$ and $B$ be two isomorphic connected induced subgraphs of~$D:=Y_k$.
Let $V_1,V_2,V_3$ be the three vertex sets as in the proof of Lemma~\ref{lem_DeltaBipThenDirectedCycle} and let $\Delta_1,\Delta_2,\Delta_3$ be the corresponding reachability digraphs.
Let $\alpha$ be an isomorphism from $A$ to~$B$.
It is straightforward to see that $(VA\cap V_i)^\alpha$ is precisely the intersection of~$VB$ with a $V_j$.
So we may assume that $(VA\cap V_i)^\alpha=VB\cap V_i$.
If $A$ and $B$ have at most six vertices, then it is easy to see that every isomorphism from $A$ to $B$ extends to an automorphism of~$D$.
So we may assume that there is at least one $V_i$, say $V_1$, that contains at least three vertices of~$A$.
Then both subdigraphs $\Delta_1\cap A$ and $\Delta_3\cap A$ are connected subdigraphs.
Let $\Delta_1'$ be a minimal subdigraph isomorphic to a $CP_l$ with $l\leq k$ such that $A\cap \Delta_1=A\cap \Delta_1'$.
By replacing $B$ by $B^\gamma$, for an automorphism $\gamma$ of~$D$, we may assume that also $B\cap\Delta_1=B\cap\Delta_1'$ holds.
Since $CP_l$ is a C-homogeneous bipartite graph, we can extend every isomorphism from $\Delta_1'\cap A$ to $\Delta_1'\cap B$ to an automorphism of $\Delta_1'$ and hence, in particular, the restriction of $\alpha$.
Let $\alpha'$ be the automorphism of~$\Delta_1'$ that extends the above restriction of~$\alpha$.
Let $V_3'\sub V_3$ be the set of those vertices which are not adjacent to all vertices of $\Delta_1'$.
As each vertex in $V_3'$ is uniquely determined by two non-adjacent vertices one of which lies in $V_1$ and the other in $V_2$, $\alpha'$ has precisely one extension $\beta$ on~$D':=D[V\Delta_1' \cup V_3']$.
By the construction of $\beta$ it is easy to see that the restriction of $\alpha$ to $D'$ is again an isomorphism from $A\cap D'$ to $B\cap D'$ and is equal to the restriction of~$\beta$ to~$A\cap D'$.
Since all vertices of $A\cap (V_3\sm V_3')$ are adjacent to all vertices of $A\cap (V_1\cup V_2)$ and since the same holds for $B$ instead of~$A$, $\beta$ can be extended to an automorphism of~$D$ whose restriction to~$A$ is $\alpha$.
\end{proof}

\providecommand{\bysame}{\leavevmode\hbox to3em{\hrulefill}\thinspace}
\providecommand{\MR}{\relax\ifhmode\unskip\space\fi MR }
\providecommand{\MRhref}[2]{%
  \href{http://www.ams.org/mathscinet-getitem?mr=#1}{#2}
}
\providecommand{\href}[2]{#2}

\end{document}